\documentclass{amsart}

\usepackage{amsmath}
\usepackage{amssymb}
\usepackage{graphicx}
\usepackage{color}
\usepackage[all]{xy}
\usepackage{amsthm}
\usepackage{verbatim}

\usepackage{tikz}
\usetikzlibrary{patterns}

\usepackage[colorlinks,linkcolor=blue,citecolor=blue,pdfstartview=FitH]{hyperref}
\theoremstyle{plain}

\newtheorem{para}{}[section]
\newtheorem{thm}[para]{Theorem}
\newtheorem{prop}[para]{Proposition}
\newtheorem{lemma}[para]{Lemma}
\newtheorem{cor}[para]{Corollary}

\newtheorem{fact}[para]{Fact}

\newtheorem*{question}{Question}

\theoremstyle{remark}

\newtheorem{remark}[para]{Remark}

\theoremstyle{definition}
\newtheorem{dfn}[para]{Definition}
\newtheorem{problem}{Problem}

\newcommand{\wt}[1]{\widetilde{#1}}

\newcommand{\co}{\colon\thinspace}

\newcommand{\bound}{\partial}

\newcommand{\Z}{\mathbb{Z}}

\renewcommand{\c}{{\sf c}}

\newcommand{\f}{{\sf f}}

\newcommand{\g}{{\sf g}}

\newcommand{\h}{{\sf h}}

\newcommand{\p}{{\sf p}}

\newcommand{\s}{{\sf s}}

\renewcommand{\t}{{\sf t}}

\newcommand\calb{\mathcal{B}}

\newcommand\calh{\mathcal{H}}

\newcommand\calp{\mathcal{P}}
\newcommand\calq{\mathcal{Q}}

\newcommand\sfa{{\sf a}}
\newcommand\sfb{{\sf b}}
\newcommand\sfc{{\sf c}}
\newcommand\sfm{{\sf m}}

\newcommand\sfn{{\sf n}}
\newcommand\sfr{{\sf r}}

\newcommand\sfx{{\sf x}}

\definecolor{gray1}{gray}{0.6}

\begin{document}

\title{Hidden symmetries via hidden extensions}

\author{Eric Chesebro}
\address{Department of Mathematical Sciences, University of Montana} 
\email{Eric.Chesebro@mso.umt.edu} 

\author{Jason DeBlois}
\address{Department of Mathematics, University of Pittsburgh} \email{jdeblois@pitt.edu}

\begin{abstract}  This paper introduces a new approach to finding knots and links with hidden symmetries using ``hidden extensions'', a class of hidden symmetries defined here.  We exhibit a family of tangle complements in the ball whose boundaries have symmetries with hidden extensions, then we further extend these to hidden symmetries of some hyperbolic link complements.
\end{abstract}

\maketitle

%%%%%%%%%%%%%%%%%%%%%%%%%%%%%%%%%%

A \textit{hidden symmetry} of a manifold $M$ is a homeomorphism of finite-degree covers of $M$ that does not descend to an automorphism of $M$.  By deep work of Margulis, hidden symmetries characterize the arithmetic manifolds among all locally symmetric ones: a locally symmetric manifold is arithmetic if and only if it has infinitely many ``non-equivalent'' hidden symmetries (see \cite[Ch.~6]{Zimmer}; cf.~\cite{NeumReid}).

Among hyperbolic knot complements in $S^3$ only that of the figure-eight is arithmetic \cite{ReidFig8}, and the only other knot complements known to possess hidden symmetries are the two ``dodecahedral knots'' constructed by Aitchison--Rubinstein \cite{AiRu}.  Whether there exist others has been an open question for over two decades \cite[Question 1]{NeumReid}.   Its answer has important consequences for commensurability classes of knot complements, see \cite{ReidWalsh} and \cite{BBoCWaGT}.

The partial answers that we know are all negative.  Aside from the figure-eight, there are no knots with hidden symmetries with at most fifteen crossings \cite{Dunfield} and no two-bridge knots with hidden symmetries \cite{ReidWalsh}.  Macasieb--Mattman showed that no hyperbolic $(-2,3,n)$ pretzel knot, $n \in \Z$, has hidden symmetries \cite{MacMat}.  Hoffman showed the dodecahedral knots  are commensurable with no others \cite{Hoffman}.

Here we offer some positive results with potential relevance to this question.  Our first main result exhibits hidden symmetries with the following curious feature.

\begin{dfn}\label{the hiddens}  For a manifold $M$ (possibly with boundary) and a submanifold $S$ of $M$, a \textit{hidden extension} of a self-homeomorphism $\phi$ of $S$ is a hidden symmetry $\Phi\co M_1\to M_2$ of $M$, where $p_i\co M_i\to M$ are connected, finite-sheeted covers for $i=1,2$, that lifts $\phi$ on a component of $p_1^{-1}(S)$.\end{dfn}

We use a family $\{L_n\}$ of two-component links constructed in previous work \cite{CheDe}.  For each $n$, $L_n$ is assembled from a tangle $S$ in $B^3$, $n$ copies of a tangle $T$ in $S^2\times I$, and the mirror image $\overline{S}$ of $S$.  Figure \ref{elltwo} depicts $L_2$, with light gray lines indicating the spheres that divide it into copies of $S$ and $T$.  For $n\in\mathbb{N}$ and $m\geq 0$, we will also use a tangle $T_n\subset L_{m+n}$: the connected union of $S$ with $n$ copies of $T$.  For instance, $L_2$ contains a copy of $T_1$ (which is pictured in Figure \ref{SUT} below) and of $T_2$.  

Upon numbering the endpoints of $T_n$ as indicated in Figure \ref{SUT}, order-two even permutations determine \textit{mutations}: mapping classes of $\partial (B^3-T_n)$ induced by $180$-degree rotations of the sphere obtained by filling the punctures.

\begin{figure}%[ht]

\setlength{\unitlength}{.1in}

\begin{picture}(40,10)
\put(0,0) {\includegraphics[width=4in]{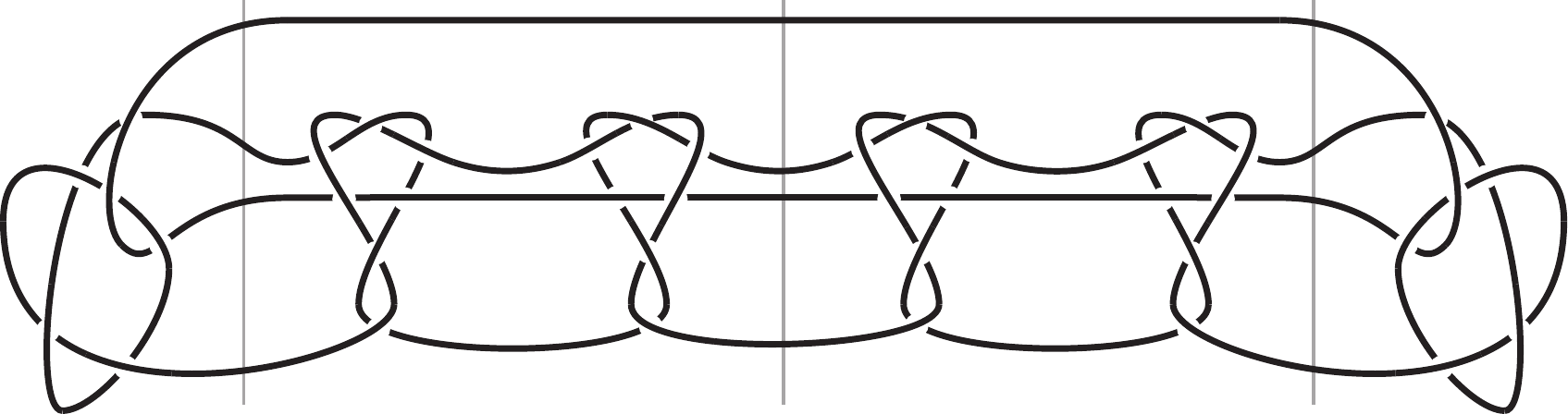}}
\put(12.5,0.25) {$T$}
\put(2,9) {$S$}
\end{picture}

\caption{The link $L_2$} \label{linkspic}
\label{elltwo}
\end{figure}

\theoremstyle{plain}
\newcommand\HiddenExtension{For $n\in\mathbb{N}$, the mutation of $\partial(B^3 - T_n)$ determined by $(1\,3)(2\,4)$ has a hidden extension over a cover of $B^3-T_n$ and for any $m\in\mathbb{N}$, taking $T_n\subset L_{m+n}$, a hidden extension over a cover of $S^3-L_{m+n}$.}
\newtheorem*{HiddenXtnProp}{Theorem \ref{hidden extension}}
\begin{HiddenXtnProp}\HiddenExtension \end{HiddenXtnProp}

In particular, this gives the first proof that the $S^3-L_{m+n}$ have hidden symmetries.  Its heart is the fact that though $(1\,3)(2\,4)$ does not extend over $S^3-L_{m+n}$, it is represented by an isometry of the totally geodesic $\partial(B^3-T_n)$ that is induced by an isometry $\sfm_1^{(n)}$ of $\mathbb{H}^3$ in the \textit{commensurator} (see eg. \cite[p.~274]{NeumReid}) of the group $\Gamma_{m+n}$ uniformizing $S^3-L_{m+n}$.  Lemma \ref{Hn} asserts the analogous fact for the group $\Delta_n$ uniformizing $B^3-T_n$, which implies the other assertion of Theorem \ref{hidden extension}.  

In Section \ref{matching} we attack the same problem on the same examples, but from a different direction.  The idea in this section is to produce hidden symmetries without prior knowledge of an orbifold cover such as was used in Theorem \ref{hidden extension}. Instead we leverage the decomposition of $L_n$ into tangle complements, producing explicit hidden extensions of the mutation over covers of these and solving a gluing problem to piece them together to produce a hidden symmetry of $L_n$.  One nice byproduct of this approach is an explicit description of the hidden symmetry.  We show: 

\newcommand\GlueCovers{For each $n\in\mathbb{N}$ there is an $11$-sheeted cover $N_n\to B^3-T_n$ and a hidden extension $\Psi \co N_n \to N_n$ of the mutation $(1\,3)(2\,4)$ acting on $S^{(n)}-T_n$.  Moreover, for each $m\in\mathbb{N}$, $\Psi$ extends to a hidden symmetry of an $11$-sheeted cover of $S^3-L_{m+n}$ that contains $N_n$.}
\newtheorem*{GlueCvrsProp}{Theorem \ref{glue covers}}
\begin{GlueCvrsProp}\GlueCovers\end{GlueCvrsProp}

Given that we are motivated by hidden symmetries of knot complements, the following question is natural:

\begin{question} Is there a knot $K$ in $S^3$ and a hidden symmetry of $S^3-K$ that is a hidden extension of a symmetry of some surface in $S^3-K$?\end{question}

In fact as the referee has pointed out, one might ask this about the known examples with hidden symmetries.  While it seems unlikely that the figure-eight knot complement has hidden extensions, given the classification of incompressible surfaces there (see \cite[\S 4.10]{Th_notes}), we have no corresponding conjecture about the dodecahedral knot complements.  This would be interesting to know.

Another tantalizing possibility arises from the observation that each tangle $T_n$ also lies in many knots in $S^3$ which are distinct from the 3 known examples of knots with hidden symetries.  If an analog of Theorem \ref{glue covers} could be proved for any such knot it would give a new example whose complement has hidden symmetries.  We have ruled out many possibilities using a criterion given in \cite[Corollary 2.2]{ReidWalsh}: a knot complement with hidden symmetries has cusp field $\mathbb{Q}(i)$ or $\mathbb{Q}(\sqrt{-3})$.  This condition can be easily checked with with SnapPy \cite{SnapPy} and Snap (see \cite{CGHN}).  We suspect that there is a reason the $T_n$ cannot lie in knots whose complements have hidden symmetries, and intend to study this further.  

We conclude the introduction with two related problems.
\begin{problem}  Classify tangles in the ball with complements whose boundary has a symmetry with hidden extension.\end{problem}

\begin{problem}  Given a tangle $T$ in the ball $B^3$, and a symmetry $\phi$ of $\partial B^3-T$ with a hidden extension across a cover of $B^3-T$, classify the links $L$ containing $T$ such that the hidden extension of $\phi$ extends to a cover of $S^3-L$.\end{problem}

%%%%%%%%%%%%%%%%%%
\section{Existence of hidden extensions}\label{extension}

The goal of Section \ref{tanglerator} is to describe the tangle complements $B^3-T_n$ from both the topological and geometric perspectives, by collecting relevant definitions and results scattered throughout \cite{CheDe} and re-assembling them here in a more helpful order.  In this sub-section we merely summarize geometric details, referring the interested reader to \cite{CheDe} for proofs.  In Section \ref{prooferator} we prove Theorem \ref{hidden extension}.

\begin{figure}
\centering
\setlength{\unitlength}{.1in}
\begin{picture}(30,18)
\put(0,2){\includegraphics[width=2.8in]{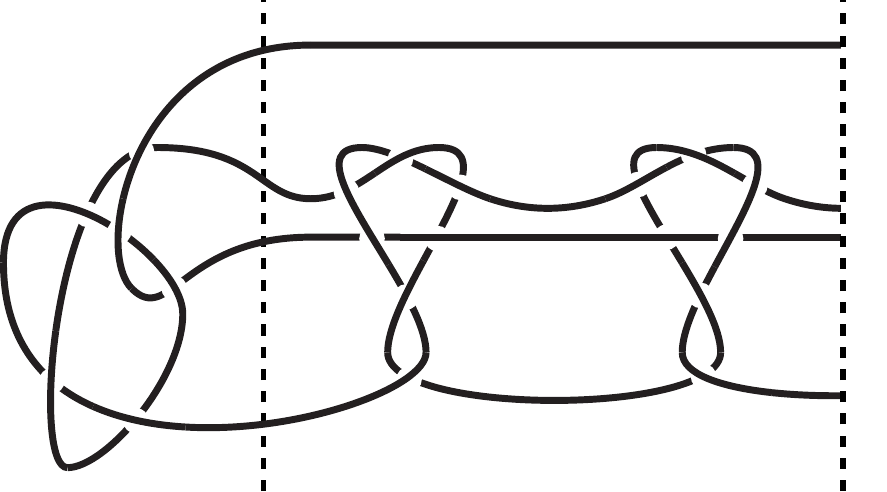}}
\put(8.8,14.8){$2$}
\put(8.8,12){$3$}
\put(8.8,8.8){$4$}
\put(8.8,4.7){$1$}
\put(27.8,16){$2$}
\put(27.8,11){$3$}
\put(27.8,9){$4$}
\put(27.8,4.5){$1$}
\put(7.5,0){$S^{(0)}$}
\put(26,0){$S^{(1)}$}
\put(1,15){$S$}
\put(16.5,2){$T$}
\end{picture}
\caption{The tangle $T_1\subset B^3$ can be decomposed along a sphere $S^{(0)}$ into a tangle $S \subset B^3$ and a tangle $T\subset S^2 \times I$.}
\label{SUT} \label{SandT}
\end{figure}
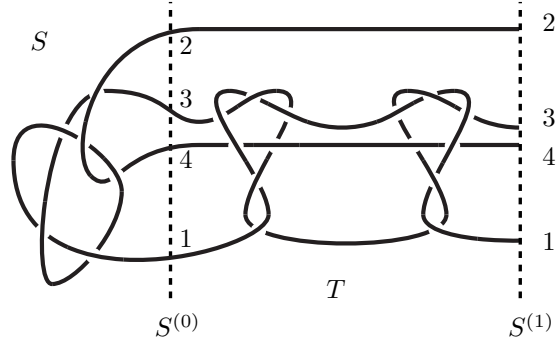

%%%%%%%%%%%%%%%%%%%%%%%%%
\subsection{The topology and geometry of $T_n$}\label{tanglerator} The solid lines in Figure \ref{SandT} describe a two string tangle $T_1 \subset B^3$.  $\bound B^3$ is shown as a dotted line labeled $S^{(1)}$.  There is an additional sphere $S^{(0)}$ shown in the figure.  If we cut $(B^3,T_1)$ along $S^{(0)}$ we obtain a pair of tangles $(B^3,S)$ and $(S^2 \times I, T)$.  Orienting $I$ so that $S^{(0)} = S^2\times\{0\}$, we let $\partial_- T = T\cap S^{(0)}$ and $\partial_+ T = T\cap S^{(1)}$.

Let $r_T \co (S^2 \times I,T) \to (S^2\times I,T)$ be the reflection homeomorphism visible in Figure 1, and let $T_0$ be the subtangle of $T$ that lies to the left of the fixed point set of $r_T$.  That is, $T_0 = T\cap S^2\times[0,1/2]$.  Reparametrizing the underlying interval, we may also regard $T_0$ as a tangle in $S^2\times I$.

Proposition \ref{omnidef} below, which combines parts of Propositions 2.7, 2.8, and 3.7 of \cite{CheDe}, introduces geometric models for the complements of the tangles $S$, $T_0$ and $T$.  There and henceforth, we work with the upper half space model for $\mathbb{H}^3$ and use the standard representation of $\text{Isom}(\mathbb{H}^3)$ as a $\Z_2$ extension of $\mathrm{PSL}_2(\mathbb{C})$.  If ${\sf d} \in \mathrm{PSL}_2(\mathbb{C})$, we write $\overline{{\sf d}}$ for the matrix whose entries are the complex conjugates of the entries of ${\sf d}$.  When we apply this operation to each element of a subgroup $\Gamma < \mathrm{PSL}_2(\mathbb{C})$ we obtain a subgroup denoted by $\overline{\Gamma}$. 

For a Kleinian group $\Gamma$, we denote the convex core of $\mathbb{H}^3/\Gamma$ as $C(\Gamma)$.  We will use the term {\em natural map} as in \cite{CheDe} (see below Definition 3.1 there) to refer to the restriction to $C(\Lambda)$ of the orbifold covering map $\mathbb{H}^3/\Lambda \to \mathbb{H}^3/\Gamma$, for $\Lambda<\Gamma$.  Because the limit set of $\Gamma$ contains that of $\Lambda$, the natural map takes $C(\Lambda)$ into $C(\Gamma)$.

The geometric models for $B^3-S$ and $(S^2\times I)-T_0$ described in parts (\ref{M_S}) and (\ref{M_T0}) of Proposition \ref{omnidef} are hyperbolic $3$-manifolds with totally geodesic boundary produced by pairing certain faces of the right-angled ideal octahedron and cuboctahedron, respectively, but they are described in the Proposition as convex cores of the quotients of $\mathbb{H}^3$ by the groups generated by the face-pairing isometries.  The equivalence of these two forms of description is proved in Lemma 2.1 of \cite{CheDe}.

\begin{prop}\label{omnidef}
\begin{enumerate}  \item\label{M_S}  For $\s= \left(\begin{smallmatrix} 1 & 0 \\ -1 & 1 \end{smallmatrix}\right)$ and $\t= \left(\begin{smallmatrix} 2i & 2-i \\ i & 1-i \end{smallmatrix}\right)$, $\Delta_0 = \langle\s,\t\rangle$ is a Kleinian group, and there is a homeomorphism $f_S\co M_S \doteq B^3-S\to C(\Delta_0)$.%\begin{align*}
%  \s &= \left(\begin{smallmatrix} 1 & 0 \\ -1 & 1 \end{smallmatrix}\right) & 
%  \t &= \left(\begin{smallmatrix} 2i & 2-i \\ i & 1-i \end{smallmatrix}\right) \end{align*}
\item\label{M_T0}  For $\f$, $\g$ and $\h$ below, $\Gamma_{T_0} = \langle \f,\g,\h \rangle$ is a Kleinian group, and there is a homeomorphism $f_{T_0}\co M_{T_0} \doteq (S^2\times I)-T_0\to C(\Gamma_{T_0})$.\begin{align*}
  \f\ &=\ \left( \begin{smallmatrix} 1 & 0 \\ -1 & 1 \end{smallmatrix} \right) &
  \g\ &=\ \left(\begin{smallmatrix} -1+i\sqrt{2} &  1-2i\sqrt{2} \\ -2 & 3-i\sqrt{2} \end{smallmatrix}\right) &
  \h\ &=\ \left(\begin{smallmatrix} 2i\sqrt{2} & -3-i\sqrt{2} \\ -3+i\sqrt{2} & -3i\sqrt{2} \end{smallmatrix}\right)\end{align*}
\item\label{T-ish}  For $\sfc= \left(\begin{smallmatrix} 1 & i\sqrt{2} \\ 0 & 1 \end{smallmatrix}\right)$, $\Gamma_T = \left\langle \Gamma_{T_0},\sfc^{-2}\overline{\Gamma}_{T_0}\sfc^2\right\rangle$ is a Kleinian group, and there is a homeomorphism $f_T\co M_T \doteq (S^2\times I)-T\to C(\Gamma_T)$ satisfying:\begin{itemize}
	\item composing the inclusion $M_{T_0}\to M_T$ with $f_T$ yields $f_{T_0}$; and
	\item for $r_T$ as above, $f_T\circ r_T\circ f_T^{-1}$ is induced by $x\mapsto \sfc^{-2}\bar{x}\sfc^2$.\end{itemize}
\item The intersection $\Delta_0 \cap \Gamma_T$ is a Fuchsian group $\Lambda$ stabilizing the hyperplane $\calh=\mathbb{R} \times (0,\infty)$ of $\mathbb{H}^3$.  This is the intersection of the convex hulls of the limit sets of $\Delta_0$ and $\Gamma_T$, and the natural maps from $\calh/\Lambda$ to $C(\Delta_0)$ and $C(\Gamma_T)$ map to totally geodesic boundary components.
\item\label{naturelle} The image of the natural map $\calh/\Lambda\to C(\Gamma_T)$ is the image of $\bound_-M_T\doteq (S^2 \times \{0\})-T$ under $f_T$.  The same holds with each instance of $T$ here replaced by $T_0$.

For the homeomorphism $j \co (\bound B^3, \bound S) \to (S^2 \times \{0\}, \bound_- T)$ such that $(B^3,S) \cup_j (S^2 \times I,T) \cong (B^3, T_1)$, $f_T\circ j\circ f_S^{-1}\co \partial C(\Delta_0)\to C(\Gamma_T)$ factors through $\calh/\Lambda$ as the composition of a natural map with the inverse of another.\end{enumerate}
\end{prop}

We now turn back to topology and give an inductive definition of the tangles $T_n$, assembling $(B^3,T_n)$ from a single copy of $(B^3,S)$ and $n$ of $(S^2\times I, T)$ for each $n\in\mathbb{N}$, using  $T_1$ as pictured in Figure \ref{SandT} as the base case.  Numbering the points of $(S^{(0)},\bound T_1)$ and $(S^{(1)},\bound T_1)$ as shown in the figure, let $(S^2 \times \{1\},\bound_+ T)$, $(S^2 \times \{0\},\bound_- T)$, and $(\bound B^3,\bound S)$ inherit numberings from their inclusions to these spheres.  Note that the resulting numbering of $(S^2\times\partial I,\bound T)$ is $r_T$-invariant.

Now for $n>1$, assume for $1\leq k< n$ that tangles $T_k \subset B^3$ with labeled endpoints are defined, and, for $k>1$, inclusions $(B^3,T_{k-1})\hookrightarrow (B^3,T_k)$ and $\iota_k \co (S^2\times I,T) \to (B^3,T_{k})$, such that:\begin{itemize}
	\item $(B^3,T_k) = (B^3\cup \iota_k(S^2\times I),T_{k-1}\cup\iota_k(T))$; 
	\item $\iota_k$ preserves labels on $\partial_+T$; and
	\item the included image of $B^3$ intersects $i_k(S^2\times I)$ in a sphere $S^{(k)}$, with $(S^{(k)},S^{(k)}\cap T_k) =  (\bound B^3,\bound T_{k-1}) = \iota_k(S^2\times \{1\},\partial_+ T)$.  
\end{itemize}
Define $T_{n} \subset B^3$ as the quotient of the disjoint union $(B^3,T_n) \sqcup (S^2 \times I,T)$ by identifying $\iota_{n-1}(x,1)$ to $(x,0)$ for each $x \in S^2$; let the inclusion of $(B^3,T_{n-1})$ and $\iota_n\co (S^2\times I,T)\to (B^3,T_n)$ be induced by the respective inclusions into the disjoint union; and label the endpoints of $T_n$ coherently with $T\cap (S^2\times\{1\})$ using $\iota_n$.  It is clear by construction that the inductive hypothesis applies to $(B^3,T_n)$.

%Let $S^{(n)}$ be the image of $S^2 \times \{1\}$ under $\iota_{n}$.  Number the points of $(\bound B^3,\bound T_n)$ by pushing labels forward with $\iota_n$.  For $0 \leq i <n$, the inclusion $\iota_{n-1}$ gives spheres $S^{(i)}$ in $B^3$ which together cut $(B^3, T_n)$ into one copy of $(B^3, S)$ and $n$ copies of $(S^2\times I, T)$.  

Having topologically described the $T_n$, our next order of business is to give geometric models for their complements; that is to describe hyperbolic manifolds with totally geodesic boundary homeomorphic to the $B^3-T_n$. In parallel with our topological description of $T_n$, these are assembled from copies of the geometric models described in Proposition \ref{omnidef}.  To this end, we define:
\begin{align*}
\Gamma_T^{(j)} &= \sfc^{-2(j-1)} \Gamma_{T} \sfc^{2(j-1)} & \Lambda^{(j)} &= \c^{-2j}\Lambda \c^{2j} & F^{(j)} &= \c^{-2j}(\calh)/\Lambda^{(j)}
\end{align*}
Note for each $j$ that $C(\Gamma_T^{(j)})$ is isometric to $C(\Gamma_T)$, so it is just a copy of $M_T$, and $F^{(j)}$ is isometric to $\calh/\Lambda$.  Now with $\Delta_0$ as in Proposition \ref{omnidef}, for $n \geq 1$ let 
\[ \Delta_n = \left\langle \Delta_0, \Gamma_T^{(1)}, \hdots,\Gamma_T^{(n)}\right\rangle. \]
(In \cite{CheDe}, $\Delta_0$ is denoted as $\Gamma_S$ and $\Delta_n$ as $\Gamma_-^{(n)}$.)  The consequence of Propositions 3.10 and 3.12 of \cite{CheDe} below shows that $C(\Delta_n)$ is a geometric model for $B^3-T_n$.
 
\begin{prop}\label{tangle structure} 
For each $n\in\mathbb{N}$ there is a homeomorphism $f_n \co B^3-T_n \to C(\Delta_n)$.  Moreover, the natural map $C(\Delta_{n-1}) \to C(\Delta_{n})$ is an isometric embedding, and there is another, $\iota_n \co C(\Gamma_T) \to C(\Delta_n)$ factoring through an isometry $C(\Gamma_T)\to C(\Gamma_T^{(i)})$, such that for $n>1$ the following diagrams commute.
\[ \xymatrix{
B^3-T_{n-1} \ar[r]^{f_{n-1}} \ar[d] & C(\Delta_{n-1}) \ar[d] & M_T \ar[r]^{f_T} \ar[d]_{\iota_n} & C(\Gamma_T) \ar[d]^{\iota_n} \\
B^3-T_{n} \ar[r]_{f_{n}} & C(\Delta_{n}) & B^3-T_n \ar[r]_{f_n} & C(\Delta_n)
} \]
This also holds for $n=1$, taking $f_0 \doteq f_S\co M_S\to C(\Delta_0)$ at the top left.

The natural map $F^{(j)} \to f_n(S^{(j)}-\bound T_j)$ is an isometry onto a totally geodesic surface in $C(\Delta_n)$ when $0 \leq j \leq n$; in particular, $F^{(n)}$ is isometric to $\partial C(\Delta_n)$.\end{prop}

Our final task in comprehensively describing the $T_n$ is to translate the tangle endpoint labeling to the geometric setting, yielding a labeling of cusps of $F^{(n)}$, or equivalently, of parabolic conjugacy classes in $\Lambda^{(n)}$.  We begin below by listing representatives for the parabolic conjugacy classes in $\Lambda$ as words in $\Gamma_S$ and $\Gamma_T$.
\begin{align*}
 & \p_1 = \s^{-1} = \f^{-1} \\
 & \p_2 = \s\t\s\t^{-2} = \f\g^{-1}\f^{-1}\h^{-1}\g \\
 & \p_3 = (\t\s\t)\s^{-1}(\t\s\t)^{-1} = (\h^{-1}\f\g)^{-1}\g^{-1}(\h^{-1}\f\g), \\
 & \p_4 = \p_1\p_2\p_3^{-1} 
 \end{align*}
 A calculation shows that 
\begin{align*}
\p_1&=\left( \begin{smallmatrix} 1&0\\1&1\end{smallmatrix} \right) &
\p_2&=\left( \begin{smallmatrix} -1&5\\0&-1\end{smallmatrix} \right) &
\p_3&=\left( \begin{smallmatrix} -14&25\\-9&16\end{smallmatrix} \right) &
\p_4&=\left( \begin{smallmatrix} 29&-45\\20&-31\end{smallmatrix} \right) 
\end{align*}
From Lemma 2.4 of \cite{CheDe} we obtain the next proposition.
\begin{prop} \label{boundary} %\label{surface}  
For any $n\in\mathbb{N}$, $j\leq n$, and $k\in\{1,2,3,4\}$, the parabolic conjugacy class in $\Lambda^{(j)}$ which corresponds to the point labeled $k$ in $S^{(j)}$ is represented by $\p_k^{(j)} =\sfc^{-2j}\p_k\sfc^{2j}$. Also $\Lambda^{(j)}$ is generated by any three of the $\p_k^{(j)}$'s.
\end{prop}

We finish by giving a geometric model for the mutation with a hidden extension.  The result below follows from Proposition \ref{boundary} above and Lemma 5.5 of \cite{CheDe}.  

\begin{lemma}\label{mutator}  Let $\sfm_1 = \left(\begin{smallmatrix} -3 & 5 \\ -2 & 3 \end{smallmatrix}\right)$.  For each $n\geq 0$, $\sfm_1^{(n)}\doteq \sfc^{-2n}\sfm_1\sfc^{2n}$ normalizes $\Lambda^{(n)}$ and induces a cycle representation $(1\,3)(2\,4)$ on the the four cusps of $F^{(n)}$, where each cusp is numbered according to its corresponding parabolic isometry $\p_j^{(n)}$.
\end{lemma}

%%%%%%%%%%%%%%%%%%%%%%%%%%
\subsection{The proof of existence}\label{prooferator}  The key new tool we need to prove Theorem \ref{hidden extension} is a discrete group containing both $\Delta_n$, with finite index, and also the isometry $\sfm_1$ that induces the mutation $(1\,3)(2\,4)$.  (The group $G_{m+n}$ of \cite[Lemma 6.2]{CheDe} plays this role for the group $\Gamma_{m+n}$ uniformizing $S^3-L_{m+n}$, by \cite[Prop.~6.3]{CheDe}.)  We will use this with a standard argument to show there is a hidden extension.

As in Definitions 6.1 of \cite{CheDe}, let $\calb_0$ be the open half-ball in the upper half-space model of $\mathbb{H}^3$ bounded by the Euclidean hemisphere of unit radius centered at $0 \in \mathbb{C}$ and, for $k\in\mathbb{N}$, let $\calb_k$ be the Euclidean translate of $\calb_0$ centered at $k(-i\sqrt{2})$, where $i$ is the imaginary unit.  For complex numbers $z$ and $w$, refer by $z\calh+w$ to the geodesic plane $(z\mathbb{R}+w)\times(0,\infty)$.   

\begin{dfn}\label{Qn}  For an integer $n \geq 0$, define $\calq_n$ to be the polyhedron of $\mathbb{H}^3$ bounded by $\calh+i/2$, $i\calh$, $i\calh+1/2$, and $\partial \calb_k$ for $k\in\{0,1,\hdots,n\}$.  Further define:
\begin{enumerate}
\item  $\f_0$ by first reflecting in $i\calh$ and then in $i\calh+1/2$;
\item  $\sfb_0$ by first reflecting in $\calh+i/2$ and then in $\partial \calb_0$; and
\item  for $k\geq 0$, $\sfa_k$ by reflecting in $i\calh+1/2$ and then in $\partial \calb_k$.\end{enumerate}
\end{dfn}

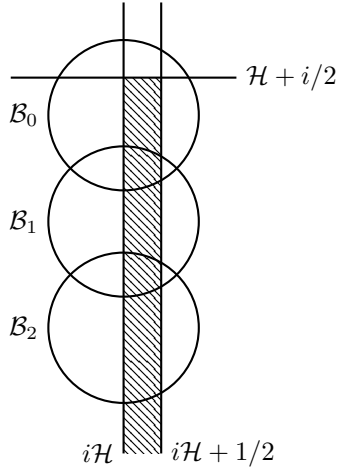
\begin{figure}[ht]
\centering
\begin{tikzpicture}[thick,scale=1]

\fill [pattern=north west lines] (0,0.5) rectangle (0.5,-4.5);

\draw (0,-4.5) -- (0,1.5);
\draw (0.5,-4.5) -- (0.5,1.5);
\draw (-1.5,0.5) -- (1.5,0.5);

\draw (0,0) circle(1);
\draw (0,-1.414) circle(1); 
\draw (0,-2.828) circle(1); 

\coordinate [label= left: $\mathcal{B}_0$] (B0) at (-1,0);
\coordinate [label= left: $\mathcal{B}_1$] (B1) at (-1,-1.414);
\coordinate [label= left: $\mathcal{B}_2$] (B2) at (-1,-2.828);
\coordinate [label= right: $\mathcal{H}+i/2$] (T) at (1.5,0.5);
\coordinate [label= left: $i\mathcal{H}$] (L) at (0,-4.5);
\coordinate [label= right: $i\mathcal{H}+1/2$] (R) at (0.5,-4.5);

\end{tikzpicture}
\caption{Bounding hyperplanes for $\calq_2$ viewed from above.}
\label{Q2}
\end{figure}
As defined, we have
\begin{align*}
\f_0 &= \left(\begin{smallmatrix} 1&1\\0&1 \end{smallmatrix}\right) & \sfb_0&=\left(\begin{smallmatrix} 0&i\\i&1 \end{smallmatrix}\right) & \sfa_0&=\left(\begin{smallmatrix} 0&-1\\1&-1\end{smallmatrix}\right) &
\sfa_1&=\left(\begin{smallmatrix} -i\sqrt{2}&1+i\sqrt{2}\\1&-1+i\sqrt{2}\end{smallmatrix}\right) .
\end{align*}
In particular, we see that $\langle \f_0, \sfa_0\rangle =\mathrm{PSL}_2(\mathbb{Z})$ contains $\sfm_1 $.

\begin{lemma}\label{Hn} For each integer $n \geq 0$, the orientation-preserving subgroup $H_n$ of the group generated by reflections in the faces of $\calq_n$ satisfies:
\begin{enumerate}
\item $H_n$ is a Kleinian group generated by $\{\f_0,\sfb_0,\sfa_0,\hdots,\sfa_n\}$.
\item $\Delta_n<H_{2n}$ with finite index.
\item The projection $\sfc^{-n}(\calh)\to\mathbb{H}^3/H_n$ factors through an isometric embedding of $\calh/\mathrm{PSL}_2(\mathbb{Z})$ onto $\bound C(H_n)$.
\end{enumerate}
\end{lemma}

\begin{proof}  Clearly $\f_0\in H_n$, $\sfb_0\in H_n$, and $\sfa_i\in H_n$ for $0\leq i\leq n$.  Let $\sfr'$ denote the reflection across $i\calh+1/2$.  It is not hard to see that $\calq_n\cup \sfr'(\calq_n)$ is a convex polyhedron in $\mathbb{H}^3$ with one face in each of $\calh+i\sqrt{2}$, $i\calh$, $i\calh+1$, $\partial \calb_k$, and $\partial \sfr'(\calb_k)$ for $k\in\{0,1,\hdots,n\}$.  The following facts can be explicitly verified:
\begin{itemize}
\item The face in $\calh+i\sqrt{2}$ meets those in $i\calh$ and $i\calh+1$ at right angles, those in $\partial\calb_0$ and $\sfr'(\partial\calb_0)$ at an angle of $\pi/3$, and no others.  The product $\sfa_0\sfb_0$ rotates by $\pi$ about an axis that bisects this face, preserving it.
\item  The face in $i\calh$ meets each of those in $\partial\calb_k$ at an angle of $\pi/2$, and none of those in $\sfr'(\partial\calb_j)$.  The element $\f_0$ takes this face to the one in $i\calh+1$.
\item  The face in $\partial\calb_k$ shares an edge with the face in $\sfr'(\partial\calb_{k'})$ if and only if $k = k'$; in this case at an angle of $2\pi/3$.  The element $\sfa_k$ takes the latter to the former.  The faces in $\partial\calb_k$ and $\partial \calb_{k-1}$ meet at an angle of $\pi/2$ for $k>0$; likewise those in $\partial\calb_k$ and $\partial\calb_{k+1}$ for $k<n$; and $\partial\calb_k\cap\partial\calb_{k'} = \emptyset$ for $k'\notin\{k-1,k,k+1\}$.\end{itemize}
Hence, $\{\f_0,\sfa_0\sfb_0,\sfa_0,\hdots,\sfa_n\}$ is a face-pairing for $\calq_n\cup\sfr'(\calq_n)$. Poincare's polyhedron theorem implies that this set of isometries generates a discrete group whose fundamental domain is $\calq_n\cup\sfr'(\calq_n)$.  By construction, this group is contained in $H_n$.  It is equal to $H_n$ because their fundamental domains have the same volume.

The numbered formulas (8) and (9) above Proposition 6.3 of \cite{CheDe} express the generators of $\Delta_0$ in terms of $\sfa_0$, $\sfb_0$, and $\f_0$ and they express $\Gamma_{T_0}$ in terms of $\sfa_0$, $\sfa_1$, and $\f_0$.  Therefore, $\Delta_0< H_n$ and $\Gamma_{T_0}<H_n$.  It can be verified directly that, for every $k$, $\sfc^{-1}\sfa_k\sfc=\sfa_{k+1}$  and that $\sfc$ commutes with $\f_0$.  It follows that $\sfc^{-k}\Gamma_{T_0}\sfc^k<H_n$ for all $k\leq n-1$.  Moreover, the second paragraph of the proof of \cite[Proposition 6.3]{CheDe} expresses the generators of $\sfc^{-2}\overline{\Gamma}_{T_0}\sfc^2$ in terms of $\sfa_1$, $\sfa_2$, and $\f_0$.  So, if $n\geq 2$, this group is also in $H_n$.  Now, by definition, we have $\Delta_n < H_{2n}$.

The polyhedron $P_n$ of \cite[Lemma 6.2]{CheDe} consists of points $(z,t)\in\calq_n$ such that the imaginary coordinate of $z$ is at least $-n\sqrt{2}$.  Every face of $P_n$ is a face of $\calq_n$ except the unique face $\mathcal{F}$ of $P_n$ contained in $\calh-n\cdot i\sqrt{2}$.  The face $\mathcal{F}$ is orthogonal to $\partial\calb_n$, $i\calh$, and $i\calh+1/2$ and does not meet any other bounding hyperplanes of $\calq_n$.  
This means that a single face of $P_n\cup\sfr'(P_n)$ contains $\mathcal{F}$, meeting only the bounding hyperplanes $i\calh$, $i\calh+1$, $\partial\calb_n$, and $\sfr'(\partial\calb_n)$.  Moreover, these intersections are all orthogonal, so $\mathcal{F}$ projects to the sole totally geodesic boundary component of the orbifold $(P_n\cup\sfr'(P_n))/H_n$.

We claim that $(P_n\cup\sfr'(P_n))/H_n=C(H_n)$.  We first show that $P_n\cup\sfr'(P_n)$ is contained in the convex hull of the limit set of $H_n$, which implies that $(P_n\cup\sfr'(P_n))/H_n$ is contained in the convex core.  Inspecting Figures 3 and 4 in \cite{CheDe}, one observes that $P_n\cup\sfr'(P_n)$ is contained in the union $\calp_1\cup\bigcup_{k=0}^{n-1} \sfc^{-k}(\calp_2)$, where $\calp_1$ and $\calp_2$ are the regular ideal octahedron and right angled ideal cuboctahedron described in Corollaries 2.2 and 2.3 of \cite{CheDe}.  Both $\calp_1$ and $\calp_2$ are the convex hulls of their ideal points, and each of these is a parabolic fixed point of $\Delta_0$ or $\Gamma_{T_0}$, respectively.  (One can show this directly, or appeal to the third-from-last paragraph of the proof of \cite[Lemma 2.1]{CheDe}.)  Since each parabolic fixed point of a Kleinian group lies in its limit set, it follows that $\calp_1\cup\bigcup_{k=0}^{n-1} \sfc^{-k}(\calp_2)$ is in the convex hull of the limit set of $H_n$.  As a subset, $P_n\cup\sfr'(P_n)$ shares this property.  On the other hand, the penultimate paragraph of \cite[proof of Lemma 2.1]{CheDe} shows that $(P_n\cup\sfr'(P_n))/H_n$ contains $C(H_n)$ and this proves our claim.  

By the above, $\sfc^{-2n}(\calh)$ projects to $\partial C(H_{2n})$ under the quotient map $\mathbb{H}^3\to\mathbb{H}^3/H_{2n}$.  By Proposition \ref{tangle structure}, the same plane projects to $\partial C(\Delta_n)$ under $\mathbb{H}^3\to\mathbb{H}^3/\Delta_n$.  It follows that the orbifold covering map $\mathbb{H}^3/\Delta_n\to\mathbb{H}^3/H_{2n}$ restricts to one $C(\Delta_n)\to C(H_{2n})$.  Since these both have finite volume, the map is finite-to-one, and hence $\Delta_n$ has finite index in $H_{2n}$.

Among all bounding hyperplanes of $\calq_n\cup\sfr'(\calq_n)$, only $i\calh$, $i\calh+1$, $\partial\calb_n$, and $\sfr'(\partial\calb_n)$ meet the hyperplane $\calh-n\cdot i\sqrt{2}$.  Each of these intersections is a right angle.  Thus, $\mathcal{F}$ is a quadrilateral and $\{\sfa_n,\f_0\}$ is an edge pairing for $\mathcal{F}$.  This implies that $\mathcal{F}/\langle \sfa_n,\f_0 \rangle$ is the boundary of $(P_n\cup\sfr'(P_n))/H_n$.

We mentioned above that $\sfc^{-n}\sfa_0\sfc^n=\sfa_{n}$ and $\f_0\sfc = \sfc\f_0$, so $\langle\sfa_n,\f_0\rangle = \sfc^{-n}\langle\sfa_0,\f_0\rangle\sfc^n$. Therefore, the projection $\calh-n\cdot i\sqrt{2} = \sfc^{-n}(\calh) \to \mathbb{H}^3/H_n$ factors through an isometric embedding of $\calh/\mathrm{PSL}_2(\mathbb{Z})$.
\end{proof}

We will use the following simple fact below, and several more times.  

\begin{fact}\label{core covers}  If $H$ has finite index in a non-elementary Kleinian group $G$ then the limit sets of $G$ and $H$ are equal, so the natural map $C(H)\to C(G)$ is an orbifold cover.\end{fact}

\begin{thm}\label{hidden extension} \HiddenExtension\end{thm}

\begin{proof}  As we mentioned in the introduction to this paper, we may view $(B^3,T_n)$ as a subset of $(S^3, L_{n+m})$. (For more rigor, compare the definitions at the beginning of this section with \cite[Definitions 3.8]{CheDe}.)  Here it is bounded by the sphere $S^{(n)}$, with the mirror image $(B^3,\overline{T}_m)$ of $T_m$ on the other side.  If the mutation $(1\,3)(2\,4)$ extended over $B^3-T_n$,  then $S^3 - L_{n+m}$ would be homeomorphic  to its mutant by $(1\,3)(2\,4)$ along $S^{(n)}$.  By Mostow-Prasad rigidity, these two links would be isometric, but by Theorem 2 of \cite{CheDe} they are not.  (In the notation of that result, $L_{m+n} = L_{(0,\hdots,0)}$ and its mutant is $L_{(0,\hdots,1,\hdots,0)}$ with the sole ``$1$'' the $(m+1)$th entry.)  This also implies it does not extend over $S^3-L_{m+n}$.

However, because $\sfm_1^{(n)}$ lies in the finite extension $H_{2n}$ of $\Delta_n$ it normalizes the normal core $\Omega_n$ of $\Delta_n$ in $H_{2n}$ and determines a self-isometry $\tilde{\Psi}$ of $\mathbb{H}^3/\Omega_n$.  This is a finite cover of $\mathbb{H}^3/\Delta_n$ which by Fact \ref{core covers} above restricts to a cover $C(\Omega_n) \to C(\Delta_n)$. 

In particular, the boundary of $C(\Omega_n)$ is totally geodesic.  One component of $\partial C(\Omega_n)$ is the quotient of $\sfc^{-2n}(\calh)$ by its stabilizer $\tilde{\Lambda}^{(n)} = \Omega_n\cap \Lambda^{(n)}$ in $\Omega_n$.  Since $\sfm_1^{(n)}$ normalizes both $\Lambda^{(n)}$ and $\Omega_n$, it normalizes $\tilde{\Lambda}^{(n)}$ and determines an isometry of $\sfc^{-2n}(\calh)/\tilde{\Lambda}^{(n)}$ lifting the one determined by $\sfm_1^{(n)}$ on $\partial C(\Delta_n)$.

A completely analogous argument applies to $S^3-L_{m+n}$, replacing $\Delta_n$ by $\Gamma_n$ from Prop.~3.12 of \cite{CheDe} and $H_{2n}$ by $G_n$ from Prop.~6.3 there.\end{proof}

%%%%%%%%%%%%%%%%%%%%%%
\section{Matching covers}\label{matching}
%%%%%%%%%%%%%%%%%%%%%%

In this section, we build an explicit hidden extension of the mutation $(13)(24)$ of $\partial (B^3-T_n)$.  To find an appropriate cover, we use the decomposition of $B^3-T_n$ along the spheres $S^{(j)}$ into one copy of $M_S$ and $n$ copies of $M_T$ and find appropriate covers of these pieces which glue together to give a cover of $B^3-T_n$ with the necessary properties.  Figure \ref{oogabooga} is a schematic depiction of how this will be done.  

For convenience, in this section we will supress the homeomorphisms $f_S$, $f_{T_0}$, and $f_T$ of Proposition \ref{omnidef} and simply make the identifications:
\begin{align*}
M_S&= C(\Delta_0) &  M_{T_0}&=C(\Gamma_{T_0}) & M_T&=C(\Gamma_T).
\end{align*}
Further, for $n\in\mathbb{N}$ and $1\leq i \leq n$ we will identify the $i$th copy $M_T^{(i)}$ of $(S^2\times I,T)$ in $(B^3,T_n)$ with $C(\Gamma_T^{(i)})$ (compare Proposition \ref{tangle structure}).

\begin{figure}
\setlength{\unitlength}{.1in}
\begin{picture}(45,29.5)
\put(0,11.5){\includegraphics{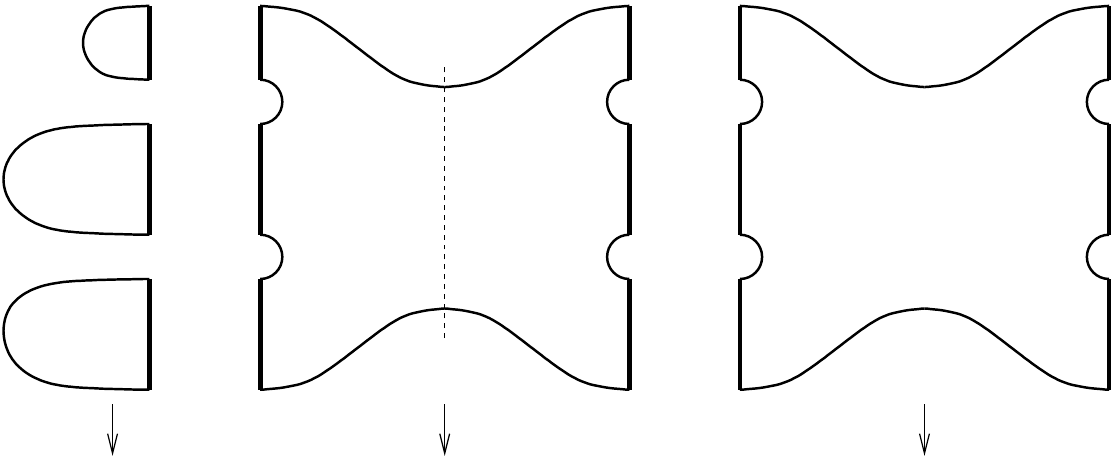}}
\put(.5,28.5){$\widetilde{M}_S$}
\put(16,28.5){$\widetilde{M}_T^{(1)}$}
\put(12.5,24){$C(\Omega_{T_0})$}
\put(36,28){$\widetilde{M}_T^{(2)}$}

\put(7,22){$\stackrel{J}{\longrightarrow}$}
\put(16.5,18.5){$\longleftrightarrow$}
\put(18.1,19.5){{\small $R_T$}}
\put(26,22){$\stackrel{R_T}{\longrightarrow}$}

\put(0,2){\includegraphics[height=.8in]{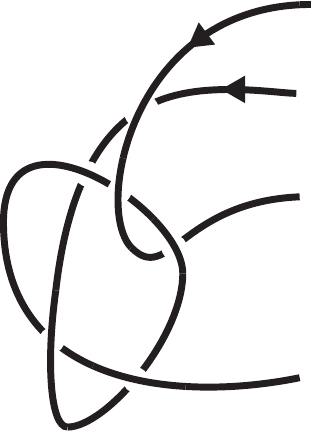}}
\put(10.3,2){\includegraphics[height=.87in]{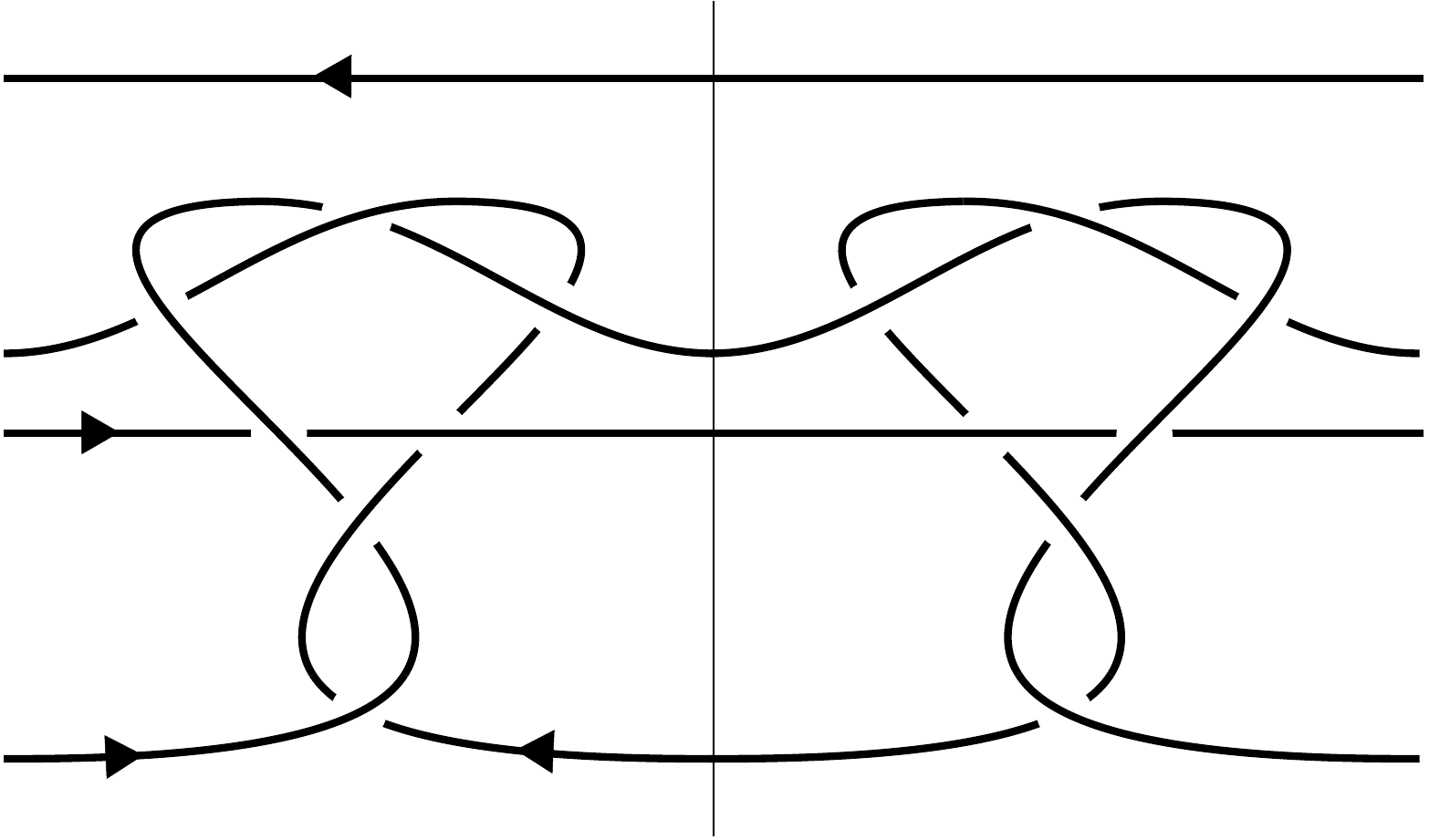}}
\put(29.6,2){\includegraphics[height=.87in]{T}}
\put(1.5,0){$M_S$}
\put(16,0){$M_T^{(1)}$}
\put(36,0){$M_T^{(2)}$}

\put(7,6){$\stackrel{j}{\longrightarrow}$}
\put(26.3,6){$\stackrel{r_T}{\longrightarrow}$}
\end{picture}

\caption{Assembling a cover of $(S^2\times I)-T_2$ that has a hidden extension of the mutation $(1\,3)(2\,4)$.}
\label{oogabooga}
\end{figure}

To produce the covers $\widetilde{M}_S$ and the $\widetilde{M}_T^{(i)}$ of Figure \ref{oogabooga} we will divide the orbifold $O_n=\mathbb{H}^3/H_n$ branched-covered by $B^3-T_n$ into pieces covered by $M_S$ and the $M_T^{(i)}$, then analyze the corresponding subgroups of $H_n$.  For $M_S$ we use $\calq_0$ from Definition \ref{Qn}: the polyhedron bounded by $\calh+i/2$, $i\calh$, $i\calh+1/2$, and $\partial \calb_0$.  Our first lemma shows that the orientation-preserving subgroup $H_0$ of the reflection group generated by $\calq_0$ contains the group $\Delta_0=\Gamma_S$ uniformizing $M_S$.  Our second uses $H_0$ and some elementary number theory to find a cover of $M_S$ with abundant symmetry.

We follow a similar strategy for $M_T$, producing a polyhedron $\calp_{T_0}$ and a group $H_{T_0}<H_n$, which Lemma \ref{the orbifold piece} shows contains the group $\Gamma_{T_0}$ uniformizing $M_{T_0}$.   We will use the permutation representation of $H_{T_0}$ given by acting on left cosets of $\Gamma_{T_0}$ to find a cover of $M_{T_0}$ with a hidden extension of $(1\,3)(2\,4)$.  Doubling this cover across a boundary component yields the model $\widetilde{M}_T$ for the $\widetilde{M}_T^{(i)}$.

\begin{lemma}\label{the other piece}  The reflection group $H_0$ (recall Lemma \ref{Hn}) has the following additional properties:
\begin{enumerate}
\item $H_0$ is a Kleinian group which contains $\Delta_0$ as a subgroup of index 12.
\item $ \displaystyle H_0 = \langle \sfa_0,\sfb_0,\f_0\,|\, \sfa_0^3 = \sfb_0^3 = (\sfb_0^{-1}\sfa_0)^2 = (\sfa_0\f_0)^2 = 1 \rangle. $
\item $\mathrm{PSL}_2(\mathbb{Z})=\mathrm{Stab}_{H_0}(\calh)$.
\item The projection $\calh\to\mathbb{H}^3/H_0$ factors through an isometric embedding of $\calh/\mathrm{PSL}_2(\mathbb{Z})$ onto $\bound C(H_0)$.
\end{enumerate}
\end{lemma}

\begin{proof} As in Lemma \ref{Hn}, $\sfr'$ is the reflection through $i\calh+1/2$.  From the proof of Lemma \ref{Hn}, we know that the collection $\{\sfa_0, \sfb_0\sfa_0, \f_0\}$ is a face-pairing for $\calq_0\cup\sfr'(\calq_0)$.  So the Poincar\'e polyhedron theorem gives the presentation above.  Aside from the fact that  $[H_0, \Delta_0]=12$, the rest of the lemma follows from the special case $n=0$ in Lemma \ref{Hn}.  But, if we compare the volume of a regular ideal octahedron in $\mathbb{H}^3$ with the volume of $P_0$, we see that $[H_0, \Delta_0]=12$. \end{proof}

\begin{lemma}\label{algebra}  There is an index five subgroup $\Omega_0 < \Delta_0$ which is normal in $H_0$.  Define $\Lambda_0 = \Omega_0\cap\Lambda$.  
\begin{enumerate}
\item $[\Lambda:\Lambda_0]=5$,
\item $\p_2,\p_4\in\Lambda_0$, and
\item $\p_1$ and $\p_3$ project to generators of $\Lambda/\Lambda_0$.
\end{enumerate} 
\end{lemma}

\begin{proof}  In the ring of Gaussian integers $5=(1+2i)(1-2i)$.  So, restricting the map $\Z[i] \to \mathbb{Z}[i]/(1+2i)$ to $\mathbb{Z}$ gives a ring epimorphism $\mathbb{Z} \to \mathbb{Z}[i]/(1+2i)$.  The quotient ring $\mathbb{Z}[i]/(1+2i)$ is isomorphic to $\mathbb{Z}/5\mathbb{Z}$ and we obtain a group epimorphism $\mathrm{PSL}_2(\mathbb{Z}[i]) \to \mathrm{PSL}_2(\mathbb{Z}/5\mathbb{Z})$ which restricts to an epimorphism $\mathrm{PSL}_2(\mathbb{Z}) \to \mathrm{PSL}_2(\mathbb{Z}/5\mathbb{Z})$.  Since $\mathrm{PSL}_2(\mathbb{Z})<H_0< \mathrm{PSL}_2(\mathbb{Z}[i])$, the restriction to $H_0$ is also onto and the kernel $\Omega_0$ of this map has index $|\mathrm{PSL}_2(\mathbb{Z}/5\mathbb{Z})|=60$ in $H_0$.

Using the explicit descriptions of $\s$ and $\t$ from Section \ref{extension}, we see that $\Delta_0$ maps onto the parabolic subgroup $\left\{\left(\begin{smallmatrix} 1 & 0 \\ * & 1\end{smallmatrix}\right)\right\}$ of $\mathrm{PSL}_2(\mathbb{Z}[i]/(1+2i))$ which has order $5$.  Hence, $\Delta_0 \cap \Omega_0$ has index five in $\Delta_0$.  Since $[H_0:\Delta_0]=12$, it follows that $[H_0:\Delta_0\cap\Omega_0] = 60$. Therefore, $\Delta_0$ contains $\Omega_0$.

Similarly, the explicit descriptions of the $\p_j$'s from Section \ref{extension} show that $\Lambda$ maps onto this same parabolic subgroup and $[\Lambda:\Lambda_0]=5$.  The final assertion is also immediate from these descriptions.\end{proof}

\begin{lemma}\label{the orbifold piece}  Let $\calp_{T_0}$ be the polyhedron bounded by $\partial\calb_0$, $\partial\calb_1$, $i\calh$ and $i\calh+1/2$.  The orientation-preserving subgroup $H_{T_0}$ of the group generated by reflections in the sides of $\calp_{T_0}$ is a Kleinian group such that
\begin{enumerate}
\item $ \displaystyle H_{T_0} = \langle \sfa_0,\sfa_1,\f_0\,|\, \sfa_0^3=\sfa_1^3 = 1, (\sfa_0\sfa_1^{-1})^2 = (\sfa_0\f_0)^2 = (\sfa_1\f_0)^2 = 1\rangle$,
\item $\bound C(H_{T_0})$ consists of a pair of totally geodesic surfaces,
\item $\mathrm{PSL}_2(\mathbb{Z})$ and $\Gamma_{T_0}$ are subgroups of $H_{T_0}$, and
\item $[H_{T_0}:\Gamma_{T_0}] = [\mathrm{Stab}_{H_{T_0}}(\calh):\Lambda] = 12$.
\end{enumerate}
 \end{lemma}

\begin{proof}  
First, recall from just before Lemma \ref{Hn} and in the proof of Lemma \ref{Hn} that Claim (3) has already been established.

Now, for visual intuition, compare $\calp_{T_0}$ with the right-angled ideal cuboctahedron $\calp_2$.  The intersection $\calp_{T_0}\cap\calp_2$ is the portion of $\calp_{T_0}$ which lies between $\calh$ and $\calh-i \sqrt{2}$.  The intersection $\calp_{T_0} \cap \calp_2$ has two additional faces contained in $\calh$ and $\calh - i\sqrt{2}$ and a single ideal vertex at $\infty$.  These additional faces are either perpendicular to or disjoint from those of $\calp_{T_0}$.

Let $\mathcal{F}$ be the face of $\calp_{T_0}\cap\calp_2$ contained in $i\calh$, $\mathcal{A}_0$ its face in $\partial\calb_0$, and $\mathcal{A}_1$ the face in $\partial\calb_1$.  Let $\sfr'$ denote reflection across $i\calh+1/2$.  A fundamental domain for $H_{T_0}$ is the union $\calp_{T_0}\cup\sfr'(\calp_{T_0})$.  By construction, $\sfa_0, \sfa_1, \f_0 \in H_{T_0}$ and these isometries determine a face-pairing for this fundamental domain.  In particular, $\sfa_0(\sfr'(\mathcal{A}_0))=\mathcal{A}_0$, $\sfa_1(\sfr'(\mathcal{A}_1))=\mathcal{A}_1$, and $\f_0(\mathcal{F})=\sfr'(\mathcal{F})$.  Thus, $H_{T_0}$ is generated by $\sfa_0$, $\sfa_1$, and $\f_0$.  Upon noting that, for each $i\in\{0,1\}$, $\mathcal{A}_i$ intersects $\mathcal{F}$ and $\mathcal{A}_{1-i}$ perpendicularly and $\sfr'(\mathcal{A}_i)$ at an angle of $2\pi/3$, the above presentation comes from the usual edge-cycle relations.

We claim that $C(H_{T_0})$ is the quotient of $\calq_{T_0}\doteq\calp_2\cap(\calp_{T_0}\cup\sfr(\calp_{T_0}))$ by the face-pairing isometries, with totally geodesic boundary.  As remarked above, the faces $\calq_{T_0}\cap\calh$ and $\calq_{T_0}\cap(\calh-i\sqrt{2})$ of $\calq_{T_0}$ intersect the others perpendicularly, so $\calq_{T_0}$ projects under the quotient map $\mathbb{H}^3/H_{T_0}$ to a suborbifold with totally geodesic boundary.  This is contained in $C(H_{T_0})$, as $\calq_{T_0}\subset\calp_2$, which by Corollary 2.3 of \cite{CheDe} is contained in the convex hull of the limit set of $\Gamma_{T_0}$ which is contained in the convex hull of the limit set of $H_{T_0}$.  Arguing as in the proof of Lemma 2.1 of \cite{CheDe} gives the reverse inclusion and hence the claim.

That $[H_{T_0}:\Gamma_{T_0}] = 12$ follows from volume considerations.  Note that $\sfa_0$ and $\sfa_1$ each preserve $\calp_2$ and fix the point $(A_0\cap A_1)\cap(\sfr(A_0)\cap\sfr(A_1))$; thus the group they generate has these properties as well.  Since $\calq_{T_0}$ contains a neighborhood in $\calp_2$ of its ideal vertex $\infty$, $\langle \sfa_0,\sfa_1\rangle$ acts freely on the set ideal vertices of $\calp_2$.  It is not hard to show directly that this action is transitive.  Since $\calp_2$ has twelve ideal vertices, its volume is twelve times that of $\calq_{T_0}$.  Since these are fundamental domains for $M_{T_0}$ and $C(H_{T_0})$ the associated cover has degree twelve.

Recall from above that the faces $\calq_{T_0}\cap\calh$ and $\calq_{T_0}\cap(\calh-i\sqrt{2})$ of $\calq_{T_0}$ project to $\bound C(H_{T_0})$.  The face-pairings of $\calq_{T_0}$ induce edge-pairings on these faces, and one checks directly that each face has its edges identified to each other.  It follows that $\bound C(H_{T_0})$ has two components.  Since $\bound M_{T_0}$ also has two components, each component of $\bound M_{T_0}$ covers twelve-to-one.  Since $\bound M_{T_0}$ has a component isometric to $F^{(0)}$, $[\mathrm{Stab}_{H_{T_0}}(\calh):\Lambda]=12$.\end{proof}

\begin{lemma}\label{permarep}  There is a homomorphism $\phi\co H_{T_0}\to S_{12}$ determined by
\begin{align*}
\phi({\sf a}_0)&=(1 \ 5 \ 9)(2 \ 6 \ 10)(3 \ 7 \ 11)(4 \ 8 \ 12)\\
\phi({\sf a}_1)&=(1 \ 8 \ 10)(2 \ 7 \ 9)(3 \ 6 \ 12)(4 \ 5 \ 11)\\
\phi({\sf f}_0)&=(1 \ 5 \ 11 \ 10 \ 3)(2 \ 7 \ 6 \ 8 \ 12)
\end{align*}
It has the following properties.
\begin{enumerate}
\item $\left| \phi(H_{T_0}) \right|=660$,
\item $\phi(\Gamma_{T_0})=\langle \phi(\h), \phi(\f) \rangle \cong \Z_{11} \rtimes \Z_5$, and
\item\label{psha} $\phi(\Lambda)=\langle \phi(\f) \rangle = \phi(\Gamma_{T_0})\cap\phi(\sfm_1\Gamma_{T_0}\sfm_1^{-1})$ is the largest subgroup of $\phi(\Gamma_{T_0})$ normalized by $\phi(\mathsf{m}_1)$.
\end{enumerate}
\end{lemma}

\begin{remark}\label{coset rep}  The homomorphism $\phi$ above is the permutation representation of $H_{T_0}$ given by its action on the left cosets of $\Gamma_{T_0}$.  This fact is not needed in the proof below or the rest of the paper.  \end{remark}

\proof
That $\phi$ is a homomorphism follows from the presentation for $H_{T_0}$ given in Lemma \ref{the orbifold piece}.  Our expressions for ${\sf a}_0, {\sf a}_1, {\sf f}_0, {\sf f}, {\sf g}, {\sf h}$, and ${\sf m}_1$ as matrices make it easy to verify the equalities
\begin{align*}
{\sf f} &={\sf a}_0 {\sf f}_0 {\sf a}_0^{-1} & 
{\sf g} &=\left( {\sf a}_0^{-1} {\sf a}_1\right) {\sf f}_0^{-1} \left( {\sf a}_0^{-1} {\sf a}_1\right)^{-1} \\
 {\sf h}&= {\sf a}_1 {\sf a}_0 {\sf f}_0^{-1} {\sf a}_1 &
 {\sf m}_1&=\left( {\sf f}_0 {\sf a}_0^{-1}\right)^2 {\sf f}_0^{-1}.
\end{align*} 
This gives
\begin{align*}
\phi({\sf f}) &= ( 2 \ 7 \ 5 \ 9 \ 3)(4 \ 6 \ 11 \ 10 \ 12) & \phi({\sf g})&=(2 \ 9 \ 12 \ 7 \ 4)(3 \ 11 \ 6 \ 8 \ 5)\\
\phi({\sf h}) &=(2 \ 12 \ 7 \ 8 \ 6 \ 3 \ 4 \ 11 \ 10 \ 9 \ 5) & \phi({\sf m}_1)&=(1 \ 8)(2 \ 12) (3 \ 4) (5 \ 11)(6 \ 9)(7 \ 10).
\end{align*}
We see that $\phi(\Gamma_{T_0})=\langle \phi(\h), \phi(\f) \rangle \cong \Z_{11} \rtimes \Z_5$, because $\phi({\sf g})=\phi({\sf fh}^{-1})$ and $\phi({\sf fhf}^{-1})=\phi({\sf h}^4)$.

Under the action of $H_{T_0}$ on $\Z_{12}$ given by $\phi$, $\Gamma_{T_0}$ is a subgroup of $\text{Stab}_{H_{T_0}}(1)$. We claim that, in fact, these groups are equal.  Let 
\begin{align*}
C &= \{ 1, 5, 9\} &
D &=\{ 8, 2, 11\}\\
E &=\{ 3, 6, 12\} &
F &= \{ 4, 7, 10\}\\
\end{align*}
and observe that $\phi \langle {\sf a}_0, {\sf a}_1 \rangle$ preserves the triples $C, D, E$, and $F$.  This gives a homomorphism $\psi \co \langle {\sf a}_0, {\sf a}_1 \rangle \to S_4$ with 
\begin{align*}
\psi({\sf a}_0)&=(D \ E \ F) &
\psi({\sf a}_1)&=(C \ D \ F).
\end{align*}
Since these two elements generate $A_4$ we have that the image of $\psi$ is the order 12 group $A_4$.  The group $\langle {\sf a}_0, {\sf a}_1 \rangle$ also acts by isometry on the polyhedron $\mathcal{P}_2$ and acts freely and transitively on its set of ideal vertices.  Hence $\left| \langle {\sf a}_0, {\sf a}_1 \rangle \right|=12$ and $\psi \co \langle {\sf a}_0, {\sf a}_1 \rangle \to A_4$ is an isomorphism.

Since $\left|\langle {\sf a}_0, {\sf a}_1\rangle\right|=12$ and $\Gamma_{T_0}$ is torsion-free, the elements of $\langle {\sf a}_0, {\sf a}_1\rangle$ make up a complete set of representatives for the left cosets of $\Gamma_{T_0}$.  If ${\sf k} \in \text{Stab}_{H_{T_0}}(1)$ then ${\sf k}={\sf a n}$, where ${\sf a} \in \langle {\sf a}_0, {\sf a}_1\rangle$ and ${\sf n} \in \Gamma_{T_0}$.  Then
\[ \phi({\sf a}) \cdot 1 = \phi({\sf a}) \phi({\sf n}) \cdot 1 = \phi({\sf k}) \cdot 1 = 1\]
and so we must also have $\psi({\sf a}) \cdot C =C$.  If we list the elements of $A_4$, we see that the only possibilities for ${\sf a}$ are the identity or ${\sf a}_0^{\pm 1}$.  Since ${\sf a}_0^{\pm 1}$ do not fix $1$, we must have ${\sf k} \in \Gamma_{T_0}$ as claimed.

We know now that $\ker \phi < \Gamma_{T_0}$, which implies that the images of distinct left cosets of $\Gamma_{T_0}$ in $H_{T_0}$ have empty intersection in $\phi(H_{T_0})$.  Therefore, $\left| \phi(H_{T_0}) \right|=55 \cdot 12=660$.

Our formulas for the ${\sf p}_j$'s in terms of ${\sf f, g}$, and ${\sf h}$ give $\phi({\sf p}_1) = \phi({\sf p}_3)=\phi({\sf f}^{-1})$ and $\phi({\sf p}_2)=\text{id}$, so $\phi(\Lambda)=\langle \phi({\sf f}^{-1}) \rangle$.  Since $\sfm_1$ normalizes $\Lambda$, $\phi({\sf m_1})$ normalizes $\phi(\Lambda)$; and indeed we have that $\phi({\sf m}_1 {\sf f} {\sf m}_1^{-1})=\phi({\sf f}^{-1})$.  On the other hand, $\phi({\sf m}_1 {\sf h}\sfm_1^{-1})$ does not stabilize $1$ while $\phi(\Gamma_{T_0})$ does, so $\phi(\Gamma_{T_0})\cap\phi(\sfm_1\Gamma_{T_0}\sfm_1^{-1})$ is properly contained in $\phi(\Gamma_{T_0})$.  This intersection contains $\phi(\Lambda)$, so since $\left| \phi(\Gamma_{T_0}) \right|=55$ the intersection equals $\phi(\Lambda)$.  It moreover follows that $\phi(\Lambda)$ is the largest subgroup of $\phi(\Gamma_{T_0})$ normalized by $\phi(\mathsf{m}_1)$.
\endproof  

Define $\Omega_{T_0} = \phi^{-1}(\phi \Lambda)$, and let $p_{T_0} \co \mathbb{H}^3/\Omega_{T_0} \to \mathbb{H}^3/\Gamma_{T_0}$ be the corresponding cover.  If we let $\theta \co H_0 \to \text{PSL}_2 (\Z[i]/(1+2i))$ be the map used in the proof of Lemma \ref{algebra} and identify $\theta(\Lambda)$ and $\phi(\Lambda)$ with the isomorphism $\phi \f \mapsto \left( \begin{smallmatrix} 1 & 0 \\ 4 & 1 \end{smallmatrix} \right)$, we see that $\theta=\phi$.  In particular, $\Lambda_0 = \Lambda\cap\ker\phi$.

\begin{lemma}\label{the T_0 cover}
The preimage of $p_{T_0}^{-1}(F^{(0)})$ has three components.  One is the inclusion-induced image of $F^{(0)}$ in $\mathbb{H}^3/\Omega_{T_0}$, which projects isometrically to $\mathbb{H}^3/\Gamma_{T_0}$.  The other two are respectively isometric to $\g(\calh)/(\g\Lambda_0\g^{-1})$ and $\g^{-1}(\calh)/(\g^{-1}\Lambda_0\g)$, each of which projects five-to-one into $\mathbb{H}^3/\Gamma_{T_0}$.
\end{lemma}

\begin{proof}  Lemma \ref{permarep} implies that $[\Gamma_{T_0}:\Omega_{T_0}]=11$.  Fact \ref{core covers} implies that $p_{T_0}$ restricts to a covering map $C(\Omega_{T_0}) \to M_{T_0}$.  By definition of $\Omega_{T_0}$, $\Lambda$ is a subgroup of $\Omega_{T_0}$ and corresponds to a component of $\bound C(\Omega_{T_0})$ that covers $\bound_- M_{T_0}$ one-to-one.  

To understand the entire preimage $p_{T_0}^{-1}(\bound_- M_{T_0})$ we pass to the cover $\widetilde{N}\to M_{T_0}$ corresponding to $\ker\phi$.  Let $f = \phi(\f)$ and $h = \phi(\h)$, so the deck group of the cover is $\phi(\Gamma_{T_0})=\langle f, h\rangle$.  $\Lambda_0$ is the stabilizer of $\calh$ in $\ker \phi$, so $\wt{S}=\calh/\Lambda_0$ is a boundary component of $\wt{N}$.  Since $[\Lambda: \Lambda_0]=5$, $\wt{S}$ is a 5-fold cover of $\bound_- M_{T_0}$.

Locate a base point $x\in\bound_- M_{T_0}$ so that elements of $\Lambda = \pi_1(\bound_- M_{T_0},x)$ are represented by loops in $\bound_- M_{T_0}$, and let $\tilde{x}\in \wt{S}$ be in the preimage of $x$.  As usual, our choice of the point $\tilde{x}$ gives an action of $\pi_1(M_{T_0},x)$ on $\wt{N}$.  By Lemma \ref{permarep}, $f \in \phi(\Lambda)$, so $f^k.\tilde{x}\in \wt{S}$ for every $k$.  Moreover, for each fixed $j$, $(h^jf^k).\tilde{x}$ and $(h^jf^{k'}).\tilde{x}=(h^jf^k).(f^{k'-k}.\tilde{x})$ occupy the same component of $\partial\widetilde{N}$ for any $k$ and $k'$.  This means that $\partial\widetilde{N}$ has eleven components covering $\bound_- M_{T_0}$.   Each component is of the form $h^j(\wt{S})$ for some $j\in\{0,\hdots,10\}$ and contains the orbit of $\tilde{x}$ under the left coset $h^j\langle f\rangle$.  Recall from the proof of \ref{permarep} that $fh = h^4f$.  Therefore, $\langle f \rangle$ acts on the components of the preimage of $\bound_- M_{T_0}$ in $\partial\widetilde{N}$.  Its three orbits are 
$$\{\wt{S}\}\quad\{h(\wt{S}),h^4(\wt{S}),h^5(\wt{S}),h^9(\wt{S}),h^3(\wt{S})\}\quad\{h^2(\wt{S}),h^8(\wt{S}),h^{10}(\wt{S}),h^7(\wt{S}),h^6(\wt{S})\}.$$
Because $C(\Omega_{T_0})$ is the quotient of $\widetilde{N}$ by the action of $\langle f\rangle$, any component of the latter two orbits projects injectively to $p^{-1}(\bound_- M_{T_0})$.  From the proof of Lemma \ref{permarep}, we know that $\phi(\g) = fh^{-1} = h^7 f$ and hence $\phi(\g^{-1}) = hf^{-1}$.  Since $h(\wt{S})$ and $h^7(\wt{S})$ lie in different $\langle f\rangle$-orbits the lemma's final claim follows.\end{proof}

Let $\wt{M}_T$ be the double of $C(\Omega_{T_0})$ across $p_{T_0}^{-1}((S^2\times\{1\})-T_0)$ and let $R_T\co\widetilde{M}_T\to\widetilde{M}_T$ be the doubling involution.  It is straightforward to show that there is an $11$-fold cover $p_T \co \widetilde{M}_T \to M_T$ that restricts on $C(\Omega_{T_0})$ to $p_{T_0}$ such that $p_T\circ R_T = r_T\circ p_T$.  Let $\bound_- \wt{M}_T=p_T^{-1}(\bound_-M_T)$.

\begin{cor}\label{the hidden extension}  The mutation of $\bound_-M_T$ determined by $(1\,3)(2\,4)$ has a hidden extension $\Psi\co\widetilde{M}_{T}\to\widetilde{M}_T$ that preserves each component of $p_T^{-1}(\bound_-M_T)$ and commutes with $R_T$.\end{cor}

\begin{proof}  
Since $\phi(\sfm_1)$ normalizes $\langle\phi(\f)\rangle$ in $S_{12}$, $\sfm_1$ normalizes $\Omega_{T_0}$.  So $\sfm_1$ induces a self-isometry $\Psi_0$ of $\mathbb{H}^3/\Omega_{T_0}$ which, by Fact \ref{core covers}, preserves $M_{T_0}$.   We claim that $\Psi_0$ preserves $p_{T_0}^{-1}((S^2\times\{j\}) - T_0)$ for $j\in \{0,1\}$.

From Lemma \ref{the orbifold piece}, we know that $\bound C(H_{T_0})$ has two components. So, under the branched cover $M_{T_0} \to C(H_{T_0})$, the images of the two components of $\bound M_{T_0}$ are distinct.  Recall that Proposition \ref{omnidef}(\ref{naturelle}) implies $\bound_- M_{T_0}=\calh/\Lambda$.  Since each component of $p_{T_0}^{-1}(\bound_-M_{T_0})$ is the quotient of a $\Gamma_{T_0}$-translate of $\calh$ by its stabilizer in $\Omega_{T_0}$, this means that the $H_{T_0}$- and $\Gamma_{T_0}$-orbits of $\calh$ are identical.  Since $\sfm_1\in\mathrm{PSL}_2(\mathbb{Z})<H_{T_0}$, the $\Gamma_{T_0}$-orbit of $\calh$ is perserved by $\sfm_1$.  This proves the claim.

Now, using the claim, we obtain an isometry $\Psi$ of $\widetilde{M}_T$ that commutes with $R_T$ and agrees with $\Psi_0$ on $C(\Omega_{T_0})$.   In the last part of the proof of Proposition 6.6 from \cite{CheDe}, we show that $\sfm_1$ does not normalize $\Gamma_T$.  Hence, $\Psi$ is a hidden symmetry of $M_T$ .  Since $\sfm_1$ normalizes $\Lambda$, $\Psi_0$ restricts to a lift of $(1\,3)(2\,4)$ on the component of $p_{T_0}^{-1}(\bound_-M_{T_0})$ which is the image of $\calh$.  Therefore, $\Psi_0$ is a hidden extension of this mutation over $C(\Omega_{T_0})$ and $\Psi$ is a hidden extension over $\widetilde{M}_T$. \end{proof}

\begin{cor}\label{over M_S}  There is a covering space $p_S\co\widetilde{M}_S\to M_S$ with degree 11 and an isometry $J \co \bound \widetilde{M}_S \to \partial_- \widetilde{M}_T$, which lifts the map $j\co\partial M_S\to \bound_-M_T$, such that $J^{-1}\Psi J$ extends across $\widetilde{M}_S$.\end{cor}

\begin{proof}  Let $\widetilde{M}_S$ be the disjoint union $M_S \sqcup \wt{N}_a \sqcup \wt{N}_b$, where $\widetilde{N}_a$ and $\widetilde{N}_b$ are copies of $C(\Omega_0)$.  Define $J$ component-wise as follows.
\begin{itemize}
 \item  Recall the isometric embedding $\iota_-^{(0)} \co F^{(0)}\to\partial M_S$ and let $\iota_+\co F^{(0)}\to\partial \widetilde{M}_T$ be the isometric embedding given by Lemma \ref{the T_0 cover}.  Define $J|_{\partial M_S}$ as $\iota_+\circ(\iota_-^{(0)})^{-1}$.
 \item  Let $\iota_a\co \calh/\Lambda_0\to\partial \widetilde{N}_a$ be the isometric embedding guaranteed by Lemma \ref{algebra} and let $\iota_{\g}\co \calh/\Lambda_0\to\partial\widetilde{M}_{T_0}$ be the composition of the isometric embedding $\g(\calh)/\g\Lambda_0\g^{-1}\to\partial_-\widetilde{M}_{T_0}$ given by Lemma \ref{the T_0 cover} with the natural isometry $\calh/\Lambda_0\to\g(\calh)/\g\Lambda_0\g^{-1}$.  Define $J|_{\partial \widetilde{N}_a}$ as $\iota_{\g}\circ\iota_a^{-1}$.
 \item Define $J|_{\partial \widetilde{N}_b}$ as $\iota_{\g^{-1}}\circ \iota_b^{-1}$ in analogy with the case above.
 \end{itemize}

To see that $J$ lifts $j$, notice first that Proposition \ref{tangle structure} implies that $j=\iota_+^{(0)}(\iota_-^{(0)})^{-1}$.  Now, for $x \in \mathbb{H}^3$, the covering map $\mathbb{H}^3/\Omega_0 \to \mathbb{H}^3/\Delta_0$ sends $\Omega_0(x)$ to $\Gamma_0(x)$ and $\calh/\Lambda_0\to F^{(0)}$ sends $\Lambda_0(x)$ to $\Lambda(x)$.  So $\iota_a$ and $\iota_b$ lift $\iota_0$.  Similarly, $\iota_{\g}$ and $\iota_{\g^{-1}}$ lift $\iota_+^{(0)}$, since $\iota_+^{(0)}$ factors as the natural composition $F^{(0)} \to \sfx(\calh)/\sfx\Lambda\sfx^{-1}\to\partial_- M_T$ whenever $\sfx \in \Gamma_T$. 

Lemma 5.8 of \cite{CheDe} implies that $\sfm_1$ normalizes $\Delta_0$.  So, the restriction of $J^{-1}\Psi J$ to $\partial M_S$ extends over $M_S$.  A calculation shows that $\g\sfm_1\g^{-1}=\left(\begin{smallmatrix}1 & -1 \\ 2 & -1\end{smallmatrix}\right)$, so $\g\sfm_1\g^{-1}$ preserves $\calh$ and $\sfm_1$ preserves $\g^{-1}(\calh)$.  For $x \in \calh$, the map $J^{-1}\Psi J|_{\partial \widetilde{N}_a}$ takes $\Omega_0(x)$ to $\Omega_0(\g^{-1}\sfm_1\g(x))$.  This map extends over $\wt{N}_a$ since $\g^{-1}\sfm_1\g(x) \in H_0$ and, by Lemma \ref{algebra}, $\Omega_0$ is normal in $H_0$.

Since $\Psi$ takes $\partial_- \widetilde{M}_T$ to itself and preserves the components covered by $\calh$ and $\g^{-1}(\calh)$, it also preserves the component covered by $\g(\calh)$.  For $x \in \calh$, the map $J^{-1}\Psi J|_{\partial \widetilde{N}_b}$ takes $\Omega_0(x)$ to $\Omega_0(\g\sfm_1\g^{-1}(x))$.  As before, this map extends over $\wt{N}_b$ because $\Omega_0$ is normal in $H_0$. \end{proof}

\begin{thm}\label{glue covers}\GlueCovers\end{thm}
\begin{proof}  
Let $\bound_+ \wt{M}_T^{(j)}= \bound \widetilde{M}_T^{(j)}-\partial_- \widetilde{M}_T^{(j)}$ and compose $R_T\co \bound_+\wt{M}_T \to \bound_-\wt{M}_T$ with marking maps to obtain isomorphisms $R_T \co \bound_+\wt{M}_T^{(j)} \to \bound_- \wt{M}_T^{(j+1)}$.  Define $N_n$ to be the adjunction space
\[ \wt{M}_S \cup_J \wt{M}_T^{(1)} \cup_{R_T} \dots \cup_{R_T} \widetilde{M}_T^{(n)}.\]
Because $J$ lifts $j$ and $R_T$ lifts $r_T$, the covering maps $p_S$ and $p_T$ determine covering spaces $p_n \co N_n \to B^3 - T_n$, which restrict to $p_S$ and $p_T$ on the factors of the adjunction space.

The hidden extension $\Psi$ is given on each of the $\widetilde{M}_T^{(i)}$ by the eponymous symmetry from Corollary \ref{the hidden extension} and on $\widetilde{M}_S$ by the extension of $J^{-1}\Psi J$ described in Corollary \ref{over M_S}.  Corollary \ref{the hidden extension} implies that  $\Psi \co N_n \to N_n$ does not descend to $B^3 - T_n$. 

Let $\overline{p}_n \co \overline{N}_m \to \overline{B^3-T_m}$ be the mirror image of the cover $p_n$ and $\overline{\Psi} \co \overline{N}_m \to \overline{N}_m$ the mirror image of $\Psi$.  Using the mirror map $\bound N_n \to \bound N_m$ to glue, we form an adjunction space $\wt{M}_{m+n} = N_n \cup \overline{N}_m$.  The covering maps $p_n$ and $\overline{p}_n$ determine a covering space $\wt{M}_{m+n} \to S^3-L_{m+n}$ and $\Psi$ and $\overline{\Psi}$ determine an isometry $\wt{M}_{m+n} \to \wt{M}_{m+n}$.  As defined, the covering space and isometry restrict to $p_n$ and $\Psi$ on $N_n \subset \wt{M}_{m+n}$.\end{proof} 

\begin{remark}  The referee has asked whether the covers of $B^3-T_n$ and $S^3-L_{m+n}$ described in Theorem \ref{glue covers} have minimal degree among those admitting hidden extensions of $(1\,3)(2\,4)$.  We suspect this is so but cannot quite prove it.  Below we prove a related but weaker assertion. Suppose $\Phi\co M_1\to M_2$ is a hidden symmetry of $M_{T_0}$ that restricts to a lift of $(1\,3)(2\,4)$ on a component $S_1$ of $p_1^{-1}(\partial_- M_T)$, where $p_1\co M_1\to M_{T_0}$ and $p_2\co M_2\to M_{T_0}$ are finite-degree connected covers of $M_{T_0}$.  If $\Phi$ is induced by some $\sfn\in H_{T_0}$ from Lemma \ref{the orbifold piece} then the $p_i$ have degree at least 11.

Let $\Gamma_1$ and $\Gamma_2$ be the finite-index subgroups of $\Gamma_{T_0}$ respectively corresponding to the $M_i$.  There exist finite-index subgroups $\Lambda_i$ of $\Lambda$, and $g_i\in\Gamma_{T_0} - \Gamma_i$ so that $S_1$ is represented in $\Gamma_1$ by $g_1\Lambda_1g_1^{-1}$ and $S_2 = \Phi(S_1)$ by $g_2\Lambda_2g_2^{-1}$.  For each $i$, the restriction of $p_i$ to $S_i$ has ${p_i}_*(g_i\lambda g_i^{-1}) = \lambda$ for $\lambda\in \Lambda_i$.  That $\Phi\co S_1\to S_2$ lifts $(1\,3)(2\,4)$ translates at the level of induced maps to 
$$ \sfm_1\lambda\sfm_1^{-1} = g_2^{-1}\Phi_*(g_1\lambda g_1^{-1})g_2 = g_2^{-1}\sfn g_1 \lambda g_1^{-1}\sfn^{-1}g_2 $$
for each $\lambda\in \Lambda_1$, since $\sfm_1\co \Lambda\to\Lambda$ is the induced map of $(1\,3)(2\,4)$.  But the centralizer of $\Lambda_1$ in $\mathrm{PSL}_2(\mathbb{C})$ is trivial, so we have $\sfn = g_2\sfm_1g_1^{-1}$.  Lemma \ref{permarep}(\ref{psha}) now implies that $\Gamma_1$ has index at least $11$ in $\Gamma_{T_0}$.\end{remark}

\subsection*{Acknowledgement}  We are grateful to the referee for asking interesting questions, and for suggesting changes that have improved the quality of this paper.

\bibliographystyle{plain}
\bibliography{hidden_bib}

\end{document}